\documentclass[amssymb, 11pt]{amsart}
\usepackage{latexsym}

\newlength{\standardunitlength}
\newtheorem{prop}{Proposition}[section]

\newtheorem{lemma}[prop]{Lemma}
\newtheorem{cor}[prop]{Corollary}
\newtheorem{theorem}[prop]{Theorem}

\begin{document}

\title[Beta approximation for the two alleles]{Beta approximation for the two alleles Moran model by Stein's method}

\author{Jason Fulman}
\address{Department of Mathematics, University of Southern California, Los Angeles, CA 90089-2532, USA }
\email{fulman@usc.edu}

\date{July 25, 2023}

\thanks{The author was funded by Simons Foundation grants 400528 and 917224, and thanks Christian D\"obler for helpful
conversations.}

\begin{abstract} In work on the two alleles Moran model, Ewens showed that the stationary distribution for the number of genes
of one type can be approximated by a Beta distribution. In this short note, we provide a sharp error term for this approximation. We
show that this example fits perfectly into D\"obler's framework for Beta approximation by Stein's method of exchangeable pairs.
\end{abstract}

\maketitle


\section{Introduction}

In work on the ``two alleles Moran model'' of populations genetic Ewens (pages 107-108 of \cite{Ewens}) is led to study the stationary
distribution $\pi$ of the Markov chain on the set $\{0,1,\cdots,2n\}$ with transition probabilities

\begin{eqnarray*}
p(i,i-1) & = & [i(2n-i)(1-v) + u i^2]/(2n)^2 \\
p(i,i+1) & = & [i(2n-i)(1-u) + v (2n-i)^2]/(2n)^2 \\
p(i,i) & = & 1-p(i,i-1)-p(i,i+1).
\end{eqnarray*} Here $0 \leq u,v \leq 1$ are parameters. 

As Ewens shows, there is an exact formula for this stationary distribution:

\[ \pi(i) = \pi(0) \frac{(2n)! \Gamma(i+A) \Gamma(B-i)}{i! (2n-i)! \Gamma(A) \Gamma(B)}. \]
Here \[ \Gamma(t)= \int_0^{\infty} x^{t-1}e^{-x}dx \] is the well-known gamma  function, $A=2nv/(1-u-v)$, $B=2n(1-v)/(1-u-v)$,
$C=2nu/(1-u-v)$, $D=2n/(1-u-v)$ and $\pi(0)=\Gamma(B) \Gamma(A+C) / [\Gamma(D) \Gamma(C)]$.

Unfortunately, this formula is hard to work with, so Ewens approximates $W$ by a Beta distribution.
More precisely, pick $I$ from the distribution $\pi$ and let $W=I/(2n)$. Then letting $v=a/(2n)$ and $u=b/(2n)$,
Ewens shows that for $a,b$ fixed and $2n$ large, $W$ is close to the Beta(a,b) distribution which has density

\[ \frac{\Gamma(a+b)}{\Gamma(a) \Gamma(b)} x^{a-1} (1-x)^{b-1} \ , \ 0<x<1 \] and $0$ else.

In this note, we use Stein's method of exchangeable pairs to compute the mean and variance of $W$ (not totally obvious from the
definition of $W$) and to give a sharp error term of order $1/n$ for Ewens' result.
We use what is known as the $d_2$ distance in the Stein's method community
(see the bottom of page 4 of \cite{DoeblerPeccati}, for example). The $d_2$ distance between random variables $X$ and $Y$ is defined as
\[ \sup_{h \in H_2} |E[h(X)]-E[h(Y)]| \] where $E$ denotes expected value and $H_2$ consists of the differentiable functions
 $h$ on $\mathbb{R}$ such that $h^{'}$ is Lipschitz continuous and $||h^{'}||_{\infty},||h^{''}||_{\infty} \leq 1$. 
Note that since $h'$ is Lipschitz, $h^{''}$ exists Lebesgue almost everywhere. The norms on $h'$ and $h''$ are the essential
supremum norms.

Our main result can be stated as follows.

\begin{theorem} \label{maintheorem}
1) The $d_2$ distance between $W$ and a Beta(a,b) random variable is at most $K(a,b)/n$, where $K(a,b)$ is an
explicit constant depending only on $a$ and $b$.  One can take $K(a,b)$ to be
\[ \frac{(9a+6b) C(a,b) + C(a+1,b+1) + (a+b) C(a+1,b+1) C(a,b)}{12} \] where $C(\cdot,\cdot)$ are defined in Theorem
\ref{Doe} below. \[ \]
2) The $d_2$ distance between $W$ and a Beta(a,b) random variable is at least
\[ \frac{ab}{4n(a+b)(1+a+b)^2}.\]
\end{theorem}

{\it Remark:} From Lemma 1.4 of \cite{DoeblerPeccati}, the Wasserstein distance can be upper bounded in terms of the $d_2$
distance. Moreover, for a Beta distribution with bounded density ($a \geq 1$ and $b \geq 1$), one can also upper
bound the Kolmogorov distance in terms of the $d_2$ distance. 
\[ \]

In Section \ref{mainresults} of this paper, we will deduce Theorem \ref{maintheorem} from a general result of  D\"obler \cite{Doebler}. The example seems quite interesting and we believe it will serve as a useful testing ground for Stein's method researchers. Indeed, it is a ``minor miracle'' that the natural exchangeable pair $(W,W')$ for our example exactly satisfies the condition \[ 4n^2 E[W'-W|W] = (a+b) \left( \frac{a}{a+b} - W \right).\]

To close the introduction, we mention two natural problems for follow-up work. First, it would be interesting to have a sharp bound for the distance between $W$ and a Beta(a,b) random variable in the Wasserstein and Kolmogorov metrics. The Wasserstein case can perhaps be studied using the methods of Goldstein and Reinert \cite{GoldsteinReinert}. Second, it would be interesting to have a multivariate generalization of our example, possibly using work on Dirichlet distributions in \cite{GanRoss} and \cite{GanRoellinRoss}.

\section{Main results} \label{mainresults}

Recall that a pair of random variables $W,W'$ is called exchangeable if the distribution
of $(W,W')$ is the same as that of $(W',W)$. We will apply the following result (a special case of the much more general 
Theorem 4.4 of D\"obler \cite{Doebler}).

\begin{theorem} \label{Doe} Let $(W,W')$ be an exchangeable pair and suppose that for a constant $\lambda>0$,
\begin{equation} \label{cond1}
\frac{1}{\lambda} E[W'-W|W] = (a+b) \left( \frac{a}{a+b} - W \right)
\end{equation} and 
\begin{equation} \label{cond2}
 \frac{1}{2 \lambda} E[(W'-W)^2|W] = W(1-W) + S
\end{equation} for a remainder term $S$.

Then the $d_2$ distance between $W$ and a Beta(a,b) random variable is at most

\[ C(a,b) E|S| + \left( C(a+1,b+1) + (a+b) C(a+1,b+1) C(a,b) \right) \cdot \frac{E|W'-W|^3}{6 \lambda},\]
where $C(\cdot,\cdot)$ are constants defined by
\begin{equation*} C(a,a) = \left\{ \begin{array}{ll}
4 & \mbox{if $0<a<1$} \\
\frac{2a \sqrt{\pi} \Gamma(a)}{\Gamma(a+1/2)} & \mbox{if $a \geq 1$}
\end{array} \right. \end{equation*}  and for $a \neq b$ by
\begin{equation*} C(a,b) = 2(a+b) \left\{ \begin{array}{ll}
\frac{\Gamma(a) \Gamma(b)}{\Gamma(a+b)}  & \mbox{if $a \leq 1, b \leq 1$} \\
a^{-1} & \mbox{if $a \leq 1, b>1$} \\
b^{-1} & \mbox{if $a>1, b \leq 1$} \\
a^{-1} b^{-1} \frac{\Gamma(a+b)}{\Gamma(a) \Gamma(b)} & \mbox{if $a>1,b>1$}.
\end{array} \right. \end{equation*} 
\end{theorem} 

We now construct the natural exchangeable pair $(W,W')$ for this example. This pair exactly satisfies Condition (\ref{cond1}) of Theorem \ref{Doe}. Moreover, the remainder term $S$ in Condition (\ref{cond2}) of Theorem \ref{Doe} turns out to be small.

To construct the pair $(W,W')$ we use the Markov chain in the first paragraph of the introduction. More precisely, since the Markov chain
is a birth-death chain, it follows that $\pi(i) p(i,j) = \pi(j) p(j,i)$ for all $i$ and $j$. This allows us
to construct an exchangeable pair $(I,I')$ as follows: choose $I \in \{0,1,\cdots,2n\}$ from $\pi$ and then obtain $I'$ by taking one step according to
the Markov chain. Rescaling by letting $W=I/(2n)$ and $W'=I'/(2n)$ gives our exchangeable pair $(W,W')$. We note that the idea of using Markov chains to construct exchangeable pairs is not new; see for instance \cite{RinottRotar} or \cite{Fulman}. 

As in the introduction, we let $a=2nv$ and $b=2nu$. 

Lemma \ref{firstsat} shows that Condition (\ref{cond1}) of Theorem \ref{Doe} is satisfied.

\begin{lemma} \label{firstsat} For $\lambda=1/(4n^2)$, we have that
\[ \frac{1}{\lambda} E[W'-W|W] = (a+b) \left( \frac{a}{a+b} - W \right).\]
\end{lemma}

\begin{proof} By the construction of the pair $(W,W')$, one has that
\begin{eqnarray*}
& & E[W'-W|W] \\
& = & \frac{1}{2n} [p(I,I+1)-p(I,I-1)] \\
& = & \frac{I(2n-I)(1-u)+v(2n-I)^2 - I(2n-I)(1-v)-uI^2}{8n^3}.
\end{eqnarray*}

This simplifies to
\begin{eqnarray*}
\frac{1}{8n^3} [2nI(-u-v) + v4n^2] & = & \frac{1}{8n^3} [2nW(-a-b) + 2na] \\
& = & \frac{1}{4n^2} [a-W(a+b)] \\
& = & \frac{(a+b)}{4n^2} \left( \frac{a}{a+b} - W \right). 
\end{eqnarray*}
\end{proof}

As a corollary of Lemma \ref{firstsat}, we compute the mean of $W$, which is not obvious from its definition. The mean agrees with that of a
Beta(a,b) random variable.

\begin{cor} \label{calcmean}
\[ E[W] = \frac{a}{a+b}. \]
\end{cor}

\begin{proof} Since $W$ and $W'$ have the same distribution, it follows from Lemma \ref{firstsat} that
\[ 0 = E[W'-W] = E[E(W'-W|W)] = E \left[ \frac{(a+b)}{4n^2} \left( \frac{a}{a+b} - W \right) \right].\]
\end{proof}

Next we calculate the variance of $W$, which will be useful in lower bounding the $d_2$ distance between $W$ and a Beta(a,b) random variable.
As with the mean, the computation of the variance of $W$ is not automatic from its definition.

\begin{prop} \label{2ndmom}
\[ Var(W) = \frac{2abn}{(a+b)^2 (2n+a(2n-1)+b(2n-1))}.\]
\end{prop}

\begin{proof} By exchangeability, $E[(W')^2-W^2]=0$. Thus
\begin{equation}\label{recur} E[E[(W')^2-W^2|W]] = 0.\end{equation} Now
\[ E[(W')^2-W^2|W] \] is proportional to 
\[ E[(I')^2-I^2|I] = p(I,I+1) \cdot ((I+1)^2-I^2) + p(I,I-1) \cdot ((I-1)^2-I^2) \] which is proportional to
\begin{eqnarray*}
& & [I(2n-I)(1-u)+v(2n-I)^2] \cdot (2I+1) \\
& & - [I(2n-I)(1-v)+uI^2] \cdot (2I-1)
\end{eqnarray*} Expanding this as a polynomial in $I$, one sees that
there is cancellation of the $I^3$ terms but not of the $I^2$ terms. Hence $E[(W')^2-W^2|W]$ is a polynomial of degree $2$ in $W$.
Thus by equation (\ref{recur}) one can express $E[W^2]$ in terms of $E[W]$, and the result follows from Corollary \ref{calcmean}.
\end{proof}

{\it Remarks:}
\begin{itemize} 
\item The variance of a Beta(a,b) random variable is equal to
\[ \frac{ab}{(a+b)^2(a+b+1)}.\] Note that Var(W) converges to this as $n \rightarrow \infty$. 

\item The method of Proposition \ref{2ndmom} can be generalized to recursively calculate higher
moments of $W$. Indeed,
let $r \geq 2$ be a positive integer. By exchangeability, $E[(W')^r-W^r]=0$. Thus
\begin{equation}\label{recur2} E[E[(W')^r-W^r|W]] = 0.\end{equation} One calculates that
\[ E[(W')^r-W^r|W] \] is a polynomial of degree $r$ in $W$. So by equation (\ref{recur2}) one can express
$E[W^r]$ in terms of $E[W^1],E[W^2],\cdots,E[W^{r-1}]$. 
\end{itemize}

Lemma \ref{secondsat} shows that Condition (\ref{cond2}) of Theorem \ref{Doe} is satisfied with a small value for the term $S$.

\begin{lemma} \label{secondsat} For $\lambda=1/(4n^2)$, we have that
\[ \frac{1}{2 \lambda} E[(W'-W)^2|W] = W(1-W) + S \]
where \[ S = \frac{1}{4n} \left[2(a+b)W^2 - (3a+b)W + a \right]. \] 
\end{lemma}

\begin{proof}
\begin{eqnarray*}
& & E[(W'-W)^2|W] \\
& = & \frac{1}{(2n)^2} [p(I,I+1) + p(I,I-1)] \\
& = & \frac{1}{(2n)^4} [I(2n-I)(1-u)+v(2n-I)^2+I(2n-I)(1-v)+uI^2]\\
& = & \frac{1}{(2n)^4} [(-2I^2+4In)+I^2(2u+2v)+I(-2nu-6nv)+4n^2v]\\
& = & \frac{8n^2}{(2n)^4} \left[ \frac{-I^2}{4n^2} + \frac{I}{2n} \right] \\
& & + \frac{8n^2}{(2n)^4} \left[ \frac{I^2}{4n^2} \frac{a+b}{2n} + \frac{I}{8n^2} (-b-3a) + \frac{a}{4n} \right] \\
& = & \frac{1}{2n^2} \left[ W(1-W) \right] + \frac{1}{2n^2} \left [\frac{a+b}{2n}W^2 - \frac{(3a+b)}{4n} W  + \frac{a}{4n} \right].
\end{eqnarray*} The theorem follows.
\end{proof}

Finally, we prove our main result.

\begin{proof} (Of Theorem \ref{maintheorem}).

For the upper bound, we apply Theorem \ref{Doe} to our exchangeable pair. Lemma \ref{firstsat}
shows that the hypotheses are met. By Lemma \ref{secondsat} and the fact that $0 \leq W \leq 1$, it follows that
\[ E|S| \leq \frac{3a+2b}{4n}.\] Since $\lambda=1/(4n^2)$ and $|W'-W| \leq 1/(2n)$, it follows that
\[ \frac{E|W'-W|^3}{\lambda} \leq \frac{4n^2}{8n^3} = \frac{1}{2n}.\] Putting these bounds together
proves the upper bound.

For the lower bound, as in \cite{GoldsteinReinert}, one would like to let $h(x)$ be the test function which is
$\frac{1}{2}(x-x^2)$ on $[0,1]$ and $0$ elsewhere. However
this function does not lie in the class $H_2$ defined in the introduction (the right hand derivative of $h$ at $0$ is not equal to
the left hand derivative of $h$ at $0$). However our random variable $W$ is supported on $[0,1]$, and as D\"{o}bler has explained,
there is a function $g$ in $H_2$ (on $\mathbb{R}$)
such that $g(x)=\frac{1}{2} x(1-x)$ on $[0,1]$. One cannot take $g(x)=\frac{1}{2}x(1-x)$ on $\mathbb{R}$ because this $g$
is not Lipschitz on all of $\mathbb{R}$. But one can take

\begin{equation*} g(x) =  \left\{ \begin{array}{ll}
h(x-2k)  & \mbox{for $x \in [2k,2k+1], (k \in \mathbb{Z})$} \\
-h(x-2k-1) & \mbox{for $x \in [2k+1,2k+2], (k \in \mathbb{Z})$}.
\end{array} \right. \end{equation*} 

So we can use the function $h(x)=\frac{1}{2}x(1-x)$ to lower bound the distance between $W$ and a Beta(a,b) random variable $Z$.
From known formulas for the mean and variance of $Z$, it follows
that \[ E[h(Z)] = \frac{ab}{2(a+b)(1+a+b)}.\] From Corollary \ref{calcmean} and Proposition \ref{2ndmom}, it follows that
\[ E[h(W)] = \frac{ab(2n-1)}{2(a+b)(2n+a(2n-1)+b(2n-1))}.\]  Thus
\begin{eqnarray*}
& & |E[h(W)]- E[h(Z)]| \\
& = & \frac{ab}{2(a+b)(1+a+b)(2n+(2n-1)a+(2n-1)b)} \\
& \geq & \frac{ab}{4n(a+b)(1+a+b)^2},
\end{eqnarray*} and the result follows.
\end{proof}

\end{document}